\documentclass[a4paper,11pt]{article}
\usepackage{amssymb,amsmath,amsthm}

\renewcommand{\title}[1]{\vspace{\fill}
\eject\addtolength{\baselineskip}{4pt}
{\bfseries\LARGE #1}\\[3mm]\addtolength{\baselineskip}{-4pt}}
\renewcommand{\author}[3]{\parbox[t]{75mm}
{\begin{center}{\scshape #1}\\[3mm] #2\\
 {\ttfamily #3} \end{center}}}

\newtheorem{thm}{\bfseries Theorem}
\newtheorem{lem}[thm]{\bfseries Lemma}        

\newtheorem{cl}[thm]{\bfseries Claim}

\newtheorem{conj}[thm]{\bfseries Conjecture}

\newtheorem{con}[thm]{\bfseries Construction}

\newcommand\lref[1]{Lemma~\ref{lem:#1}}
\newcommand\tref[1]{Theorem~\ref{thm:#1}}
\newcommand\cref[1]{Corollary~\ref{cor:#1}}
\newcommand\clref[1]{Claim~\ref{clm:#1}}
\newcommand\cnref[1]{Construction~\ref{con:#1}}
\newcommand\cjref[1]{Conjecture~\ref{conj:#1}}
\newcommand\sref[1]{Section~\ref{sec:#1}}

\begin{document}

\newcommand\cA{{\mathcal A}}
\newcommand\cB{{\mathcal B}}
\newcommand\cC{{\mathcal C}}
\newcommand\cD{{\mathcal D}}
\newcommand\cE{{\mathcal E}}
\newcommand\cF{{\mathcal F}}
\newcommand\cG{{\mathcal G}}
\newcommand\cH{{\mathcal H}}
\newcommand\cI{{\mathcal I}}
\newcommand\cJ{{\mathcal J}}
\newcommand\cK{{\mathcal K}}
\newcommand\cL{{\mathcal L}}
\newcommand\cM{{\mathcal M}}
\newcommand\cN{{\mathcal N}}
\newcommand\cP{{\mathcal P}}
\newcommand\cQ{{\mathcal Q}}
\newcommand\cR{{\mathcal R}}
\newcommand\cS{{\mathcal S}}
\newcommand\cT{{\mathcal T}}
\newcommand\cU{{\mathcal U}}
\newcommand\cV{{\mathcal V}}
\newcommand\PP{{\mathbb P}}
\newcommand\coF{{\overline F}}
\newcommand\coG{{\overline G}}
\newcommand\De{\Delta}

\begin{center}

\title{Saturating Sperner families} 
\author{D\'aniel Gerbner 
\footnotemark[5]}{
Hungarian Academy of Sciences, Alfr\'ed R\'enyi Institute\\
Mathematics, P.O.B. 127, Budapest H-1364, Hungary
}{
gerbner@renyi.hu
}\footnotetext[5]{Research supported by
    Hungarian National Scientific Fund, grant number: OTKA NK-78439}
\author{
Bal\'azs Keszegh\footnotemark[5]
}{
Hungarian Academy of Sciences, Alfr\'ed R\'enyi Institute\\
Mathematics, P.O.B. 127, Budapest H-1364, Hungary
}{
keszegh@renyi.hu
}
\author{
Nathan Lemons
}{ 
Hungarian Academy of Sciences, Alfr\'ed R\'enyi Institute\\
Mathematics, P.O.B. 127, Budapest H-1364, Hungary
}{
nathan@renyi.hu
}
\author{Cory Palmer\footnotemark[5]}
{Hungarian Academy of Sciences, Alfr\'ed R\'enyi Institute\\
Mathematics, P.O.B. 127, Budapest H-1364, Hungary}
{palmer@renyi.hu}
\author{D\"om\"ot\"or P\'alv\"olgyi\footnotemark[4]}
{Department of Computer Science\\
E\"otv\"os Lor\'and University\\
P\'azm\'any P\'eter s\'et\'any 1/C, Budapest H-1117,
Hungary}
{dom@cs.elte.hu}\footnotetext[4]{The European Union and the European Social Fund have provided financial support to the project under the grant agreement no. TÁMOP 4.2.1./B-09/1/KMR-2010-0003.}
\author{Bal\'azs Patk\'os\footnotemark[1]}
{Hungarian Academy of Sciences, Alfr\'ed R\'enyi Institute\\
Mathematics, P.O.B. 127, Budapest H-1364, Hungary}
{patkos@renyi.hu}\footnotetext[1]{Research supported by
    Hungarian National Scientific Fund, grant numbers: OTKA
    K-69062 and PD-83586}


\end{center}


\begin{quote}
{\bfseries Abstract:}
A family $\cF \subseteq 2^{[n]}$ saturates the monotone decreasing property $\cP$ 
if $\cF$ satisfies $\cP$ and one cannot add any set to $\cF$ such that property $\cP$ is still satisfied by the resulting family. We address the problem of finding the minimum size of a family saturating the $k$-Sperner property and the minimum size of a family that saturates the Sperner property and that consists only of $l$-sets and $(l+1)$-sets.

\end{quote}

\begin{quote}
\end{quote}
\vspace{5mm}


\section{Introduction}

One of the most basic  and most studied problems of extremal
combinatorics is that how many edges a (hyper)graph can have if it
possesses some prescribed property $\cP$. If this property $\cP$ is
monotone decreasing (i.e. if $G$ possesses $\cP$, then $F \subseteq
G$ implies $F$ possesses $\cP$), then there exists a ``dual''
problem to the above one: we say that a (hyper)graph $G$
\textit{saturates} property $\cP$ if it possesses property $\cP$ but
adding any (hyper)edge $E$ to $G$ would result in a graph not having
property $\cP$. The problem is to determine the minimum size that
such a saturating (hyper)graph can have. Many researchers have
dealt with this kind of problems both for graphs
\cite{BFP,B1,B2,C1,C2,C3,EHM,FK,KT,P3,PS,Tu,W1,W2} and hypergraphs
\cite{EFT,P1,P2}. To our knowledge all papers so far have considered
Tur\'an-type properties (with the exception of \cite{DPT}):
properties defined through some forbidden (hyper)graphs.

In the present paper we investigate the saturation of Sperner-type
properties. To introduce our main definitions let $k$ and $n$ be
positive integers and let $\cF \subseteq 2^{[n]}$ be a family of
sets such that,
\begin{enumerate}
\item
there do not exist $k+1$ distinct sets $F_1,...,F_{k+1} \in \cF$
that form a \textit{$(k+1)$-chain}, i.e. $F_1 \subset F_2 \subset
... \subset F_{k+1}$ holds,
\item
for every set $S \in 2^{[n]} \setminus \cF$ there exist $k$ distinct
sets $F_1,...,F_k \in \cF$ such that $S$ and the $F_i$'s form a
$(k+1)$-chain.
\end{enumerate}
A family $\cF$ is called \textit{$k$-Sperner} if it satisfies
Property 1, \textit{weakly saturating $k$-Sperner} if it satisfies 2
and \textit{(strongly) saturating $k$-Sperner} if it satisfies both.
The maximum size of a $k$-Sperner family $\cF \subset 2^{[n]}$ was
determined by Sperner \cite{S} in the special case $k=1$ and by 
Erd\H os \cite{Erd} for arbitrary $k$. 

In \sref{bounds} we will derive
bounds on $\text{sat}(n,k)$ ($\text{wsat}(n,k)$) the minimum
number of sets that strongly saturating $k$-Sperner (weakly
saturating $k$-Sperner) family $\cF \subset 2^{[n]}$ can contain. By
definition, we have $\text{wsat}(n,k) \le \text{sat}(n,k)$. The
following \textit{product construction} shows that there is an upper
bound on both of these numbers that is independent of $n$, namely
$\text{wsat}(n,k) \le \text{sat}(n,k) \le 2^{k-1}$. Let $\cF \subset
2^{[n]}$ be defined by
\[
\cF:=2^{[k-2]} \times \left\{\emptyset, [n]\setminus
[k-2]\right\}=2^{[k-2]} \cup \{F \in 2^{[n]}: [n]\setminus[k-2]
\subseteq F\}.
\]
It is easy to see that $\cF$ is indeed strongly saturating $k$-Sperner. 
It is natural to formulate the following conjecture.
\begin{conj}
\label{conj:big} For every positive integer $k$ there exists an 
$n_0=n_0(k)$ such that for any $n \ge n_0$ we have
$\text{sat}(n,k)=2^{k-1}$.
\end{conj}
It is trivial to verify that $n_0(k)=k$ for $k=1,2,3$. By giving
constructions we will prove the following two upper bounds.
\begin{thm}
\label{thm:constr} For integers $6 \le k \le n$ we have the
following inequalities:

\textbf{(i)} sat$(k,k)\le \frac{15}{16}2^{k-1}$,

\textbf{(ii)} wsat$(n,k) =O(\frac{\log k}{k}2^k)$.
\end{thm}

We will also obtain lower bounds on the size of saturating $k$-Sperner
 families. All our lower bounds will apply both for
$\text{wsat}(n,k)$ and $\text{sat}(n,k)$.

\begin{thm}
\label{thm:lower} For integers $k,c$ and $n$ we have the following
inequalities:

\textbf{(i)} $2^{k/2-1} \le \text{wsat}(n,k) \le \text{sat}(n,k)$
provided $k \le n$,

\textbf{(ii)} $\frac{2^{k+c}}{(k+c)^{c+1}} \le \text{wsat}(k+c,k)
\le \text{sat}(k+c,k)$ provided $2 \le k$ and $0 \le c$.
\end{thm}

1-Sperner families are also called \textit{antichains}. Saturating
antichains with the simplest structure are the families consisting
of all $l$-element subsets of the underlying set for any fixed $l$.
Next, one would consider antichains with two possible set sizes. If
the set sizes are consecutive integers, then these families are
called \textit{flat antichains}. Gr\"utm\"uller, Hartmann,
Kalinowski, Leck and Roberts \cite{GHKLR} proved the following
theorem.

\begin{thm} [Gr\"utm\"uller, Hartmann, Kalinowski, Leck, Roberts \cite{GHKLR}]
\label{thm:german} If $\cF \subseteq \binom{[n]}{2} \cup
\binom{[n]}{3}$ is a saturating antichain, then the following holds
\[
|\cF| \ge \binom{n}{2}-\left\lceil\frac{(n+1)^2}{8}\right\rceil.
\]
\end{thm}
The authors of \cite{GHKLR} also determined all antichains for which
equality holds. In \sref{flat} we obtain the stability version of
\tref{german}. Our proof is self-contained and is much shorter than
their proof of \tref{german}. Some of its parts generalize to saturating
flat antichains with larger set sizes. Unfortunately the lower
bounds that we obtain depend on the Tur\'an density of $K^l_{l+1}$,
the complete $l$-graph on $l+1$ vertices. We will also generalize
the construction of \cite{GHKLR}, but the lower bounds and the size
of the construction are quite far apart even if we assume that some
famous longstanding conjectures about the above-mentioned Tur\'an
densities are true. To state our stability result for a family $\cF
\subset \binom{[n]}{l}$ let us write $\Delta(\cF)=\{G \in
\binom{[n]}{l-1}: \exists F \in \cF \text{such that}\ G \subset F\}$
and $\nabla(\cF)=\{G \in \binom{[n]}{l+1}: \exists F \in \cF
\hskip 0.2truecm \text{such that}\ F \subset G\}$.

\begin{thm}
\label{thm:23} Let $\cF \subseteq \binom{[n]}{2} \cup
\binom{[n]}{3}$ be a saturating antichain of minimum size. Then the
following holds:
\[
|\cF| = \left(\frac{3}{8}-o(1)\right)n^2.
\]
Moreover, if $|\cF|=(\frac{3}{8}+o(1))n^2$, then there is a
partition $[n]=A \cup B \cup C$ with $|A|=|B|=\lfloor n/4 \rfloor$
and a matching $M$ between $A$ and $B$ such that if $\cG=\cG_2 \cup
\cG_3$ is defined by $\cG_3=\{G \in \binom{[n]}{3}: G \cap C \ne
\emptyset \hskip 0.2truecm \text{and}\ \exists m \in M \text{with}\ m \subset G\}$
and $G_2=\binom{[n]}{2} \setminus \Delta(G_3)$, then $|\cG
\vartriangle \cF|=o(n^2)$ holds and $\cG$ is a saturating
antichain.
\end{thm}

\section{Bounds on ${\rm sat}(n,k)$ and ${\rm wsat}(n,k)$}
\label{sec:bounds}

In this section, we prove \tref{constr} and \tref{lower}. We start
our investigations with an easy lemma stating that we can always
assume the empty set and $[n]$ belong to the family
$\cF$.

\begin{lem}
\label{lem:triv} If $2 \le k \le n$, then there exists a weakly (strongly) 
saturating $k$-Sperner family $\cF \subseteq
2^{[n]}$ of
minimum size such that $\emptyset$ and $[n]$ belong to $\cF$.
\end{lem}

\begin{proof}
Let $\cF$ be of minimum size with $\emptyset \notin \cF$  and let $\cF_m$  denote the minimal
sets in $\cF$. Then $\cF \setminus \cF_m \cup \{\emptyset\}$ is weakly (strongly) saturating
$k$-Sperner and its size is at most the size of $\cF$. The case of $[n]$ is
completely analogous.
\end{proof}

\begin{proof}[Proof of \tref{constr}] First we give a construction that shows that sat$(6,6) \le 30=\frac{15}{16}2^{6-1}$. We enumerate the sets according to their size:
\begin{itemize}
\item
$\emptyset$,
\item
four singletons: $\{3\},\{4\},\{5\},\{6\}$,
\item
six pairs: $\{1,2\}, \{1,3\}, \{1,4\}, \{2,3\}, \{2,4\}, \{5,6\}$,
\item
eight triples: $\{1,2,5\}, \{1,2,6\}, \{3,4,5\}, \{3,4,6\}, \{1,3,5\},
\{1,4,5\}, \{2,3,6\}, \{2,4,6\}$,
\item
six 4-tuples:
$\{3,4,5,6\},\{2,3,4,5\},\{1,3,4,6\},\{1,2,3,4\},\{1,2,5,6\},\{1,2,3,5\}$,
\item
four 5-tuples: $\{1,2,3,4,6\}, \{1,2,4,5,6\}, \{1,3,4,5,6\},
\{2,3,4,5,6\}$
\item
and $\{1,2,3,4,5,6\}$.
\end{itemize}
To see that these sets indeed form a strongly saturating 6-Sperner family, 
note that $\emptyset$, the singletons and $\{1,2\}$ form a strongly
saturating 2-Sperner family and so do $\{1,2,3,4,5,6\}$, the 5-tuples and
$\{3,4,5,6\}$. The remaining pairs together with $\{1,3,5\},
\{1,4,5\}, \{2,3,6\}, \{2,4,6\}$ form a saturating antichain as
described in \sref{flat}. Now the remaining sets form a family
isomorphic to the complements of the members of the previous family and is therefore saturating
antichain. As these four families are disjoint and lie ``one below the
other'', their union is saturating 6-Sperner. 

The following lemma finishes the proof of \tref{constr} \textbf{(i)}.
\begin{lem}
\label{lem:ind} For any integers $k,n$ such that $3\le k \le n$ we have
\[
{\text sat}(n,k) \le 2 {\text sat}(n-1,k-1).
\]
\end{lem}
\begin{proof} Let $\cF \subseteq 2^{[n-1]}$ be a strongly saturating 
$(k-1)$-Sperner family of minimum size such that (by \lref{triv}) 
$\emptyset,[n-1] \in \cF$. Then the family
$\cF'=\cF\times \{\emptyset,\{n\}\}=\cF \cup \{F \cup \{n\}:F \in
\cF\}$ is a strongly saturating $k$-Sperner subfamily of $2^{[n]}$.
Indeed, if there exists a $(k+1)$-chain $F'_1 \subset ... \subset
F'_{k+1}$ in $\cF'$, then at least $k$ out of the sets $F'_i \cap
[n-1]$ would be distinct and form a $k$-chain in $\cF$. This
contradiction shows that $\cF'$ is $k$-Sperner. To prove the
saturating property of $\cF'$ let us consider a set $G \in 2^{[n]}
\setminus \cF'$. By definition, we know that $G \cap [n-1] \notin
\cF$ holds and thus by the saturating property of $\cF$ there exists
$k-1$ sets $F_1,...,F_{k-1}$ in $\cF$ together with which $G \cap
[n-1]$ forms a $k$-chain. If $n \notin G$, then the $F_i$'s, $G$ and
$[n]$ form a $(k+1)$-chain in $\cF'$, while if $n \in G$, then 
$\emptyset,G$ and the $\{n\} \cup F_i$'s form a $(k+1)$-chain in
$\cF'$ (not necessarily in this order).
\end{proof}

To prove \textbf{(ii)} we first give a general construction. Let us write
$\cF_0=\{\emptyset\}$ and $\cF_{k}=\{[k]\}$. Furthermore, for any $1 \le
l \le k-1$ let $\cF_l \subseteq \binom{[k]}{l}$ be a family
satisfying $\nabla(\cF_l)=\binom{[k]}{l+1}$ and
$\Delta(\cF_l)=\binom{[k]}{l-1}$. Then define the family $\cF
\subseteq 2^{[n]}$ as follows:
\[
\cF=\bigcup_{l=0}^{k} \cF_l \times \{\emptyset, [n] \setminus
[k]\}=\bigcup_{l=0}^{k}\cF_l \cup \left\{F \cup ([n] \setminus [k]):
F \in \bigcup_{l=0}^{k}\cF_l\right\}.
\]
We claim that $\cF$ is a weakly saturating $k$-Sperner family.
Indeed, let us consider a set $G \in 2^{[n]} \setminus \cF$ and
define $i=|G \cap [k]|$. Observe that, by the conditions on the
$\cF_l$'s, there exist sets $F_l \in \cF_l$ for every $0\le l \le
k$, $l\ne i$ such that the $F_l$'s together with $G \cap [k]$ form a
$(k+1)$-chain. Then the $F_l$'s for $l <i$, $G$ and the sets $F_l
\cup ([n]\setminus [k])$ for $l>i$ form a $(k+1)$-chain.

It is well known \cite{CHLL} that if $l=\Theta(k)$, then there exists a
family $\cF'_l \subseteq \binom{[k]}{l}$ such that
$\Delta(\cF_l')=\binom{[k]}{l-1}$ and $|\cF'_l|=\Theta(\frac{\log
k}{k}\binom{k}{l})$ holds. Thus the size of the family $\cF_l=\cF'_l
\cup \overline{\cF'_{n-l}}$ is of the same order of magnitude and
satisfies $\Delta(\cF'_l)=\binom{[k]}{l-1},
\nabla(\cF'_l)=\binom{[k]}{l+1}$.  Use the general construction 
with $\cF_l$'s as above provided 
$k/4 \le l \le 3k/4$ and $\cF_l=\binom{[k]}{l}$ otherwise to obtain 
the family $\cF$. Then the size of $\cF$ is $2\sum_{i=0}^k|\cF_i|=O(\frac{\log k}{k}2^k)$.
\end{proof}

\vskip 0.5truecm

\begin{proof}[Proof of \tref{lower}] Let $\cF\subseteq 2^{[n]}$ be a
weakly saturating $k$-Sperner family and consider any set
$G \in 2^{[n]}\setminus \cF$. By definition, there exist $k$ distinct
sets $F_1,...,F_k \in \cF$ such that
$F_1 \subset ...\subset F_i \subset S \subset F_{i+1} \subset ...\subset F_k$
holds. Thus $G$ is a set from the interval $I_{F_i,F_{i+1}}=\{S: F_i \subseteq S \subseteq F_{i+1}$
which has size at most $2^{n-k+2}$ as $|F_{i+1}\setminus F_i| \le n-k+2$ holds by
the existence of the other $F_j$'s. We obtain that $2^{[n]}$ can be
covered by intervals of size at most $2^{n-k+2}$ and thus we have
\[
\frac{|\cF|^2}{2}2^{n-k+2} \ge \binom{|\cF|}{2}2^{n-k+2} \ge 2^n.
\]
Now \textbf{(i)} follows by rearranging.

To prove \textbf{(ii)} let us partition the intervals
$I_{F_i,F_{i+1}}$ in the proof of \textbf{(i)} according to the
$F_i$'s. The intervals belonging to the same $F_i$ may cover at most
the sets $\{S \supseteq F_i: |S \setminus F_i| \le c+1\}$. As the
number of these sets is at most $\sum_{i=0}^{c+1}\binom{k+c}{i} \le
(k+c)^{c+1}$, we obtain the inequality $|\cF|(k+c)^{c+1} \ge
2^{k+c}$ and we are done by rearranging.
\end{proof}

\vskip 0.3truecm

\tref{lower} 
\textbf{(ii)} with $c=0$ shows $wsat(k,k)=\Omega(2^k/k)$. It is one 
of the most important questions of the theory of covering codes 
whether there exist families $\cF_l \subseteq \binom{[n]}{l}$ as in the 
general construction with size $O(\binom{n}{l}/k)$. If the answer is 
positive, then one would obtain a weakly saturating $k$-Sperner family 
with size $O(2^k/k)$ via the general construction thus
$wsat(k,k)=\Theta(2^k/k)$ would follow.

\vskip 0.5truecm

Let us finish this section with some remarks on strongly saturating
$k$-Sperner families in the case when $n$ is large compared to $k$.
A family $\cF \subseteq 2^{[n]}$ is called \textit{non-separating}
if there exist $x,y \in [n]$ such that for all $F \in \cF$ we have
$x \in F \Leftrightarrow y \in F$. The family $\cF$ is separating if
it is not non-separating. Let us call a strongly saturating
$k$-Sperner family $\cF \subseteq 2^{[n]}$ \textit{duplicable} if
there exists $x \in [n]$ such that the family $\cF' \subseteq
2^{[n+1]}$ defined as $\cF':=\{F \in \cF: x \notin F\} \cup \{F \cup
\{n+1\}:x \in F\}$ is strongly saturating $k$-Sperner. Finally, a
family $\cF$ is \textit{primitive strongly saturating $k$-Sperner}
if it is separating and duplicable.

Clearly, if $2^{|\cF|} < n$, then $\cF$ is non-separating. As, by
the product construction defined in the Introduction, we know that
sat$(n,k) \le 2^{k-1}$, we obtain that any strongly saturating
$k$-Sperner family of minimum size is non-separating provided
$2^{2^{k-1}}<n$. 
Let $\cF \subseteq 2^{[n]}$ be such a family and $x,y \in [n]$ be the
elements showing the non-separating property of $\cF$. Then it is
easy to verify that the family $\cF^* \subseteq 2^{[n] \setminus
\{y\}}$ defined as $\cF^*=\{F \cap ([n] \setminus\{y\}): F \in \cF\}$
is strongly saturating $k$-Sperner and we have $|\cF|=|\cF^*|$.

Let $\cF \subseteq 2^{[n]}$ be a non-separating strongly saturating
$k$-Sperner family. We would like to prove that there is only one way to 
duplicate such a family. Formally, we claim that there do not exist 
$x,y,u,v \in
[n]$ and $F' \in \cF$ such that $x \in F \Leftrightarrow y \in F$ and $u \in F
\Leftrightarrow v \in F$ holds for all $F \in \cF$ but
exactly one of $x$ and $u$ is contained in $F'$. Indeed, if that is not the
case,  then
there would exist a set $S \subseteq [n] \setminus \{x,y,u,v\}$ such
that $S_1=S \cup \{x,u\} \notin \cF$ but at least one of $S \cup \{x,y\}$ and 
$S \cup \{u,v\}$ belongs to $\cF$ which we denote by $S_2$. Therefore there would exist sets
$F_1,...,F_k \in \cF$ that together with $S_1$ form a $(k+1)$-chain.
As the family $\cF$ does not separate $x$ and $y$ nor $u$ and $v$, all
the $F_i$'s contain either all four of $x,y,u,v$ or none of them.
But then the $F_i$'s and $S_2$ form a $(k+1)$-chain as well which
contradicts the $k$-Sperner property of $\cF$.

The above two statements yield the following lemma.

\begin{lem}
\label{lem:bign} Let $n$ and $k$ be positive integers such that
$2^{2^{k-1}} <n$ holds. Then we have
\[
\text{sat}(n,k)=\min_{n' \le 2^{2^{k-1}}}\left\{|\cF'|: \cF'
\subseteq 2^{[n']} \text{is primitive strongly saturating
$k$-Sperner}\right\}.
\]
Moreover, for any extremal family $\cF \subseteq 2^{[n]}$, there
exist integers $x \le n' \le 2^{2^{k-1}}$ and a primitive family
$\cF' \subseteq 2^{[n']}$ such that $\cF=\{F' \in \cF': x \notin
\cF'\} \cup \{F' \cup ([n] \setminus [n']): x \in F' \in \cF'\}$.
\end{lem}

\section{Saturating flat antichains}
\label{sec:flat}

In this section we consider saturating flat antichains. Let us start
with an easy lemma that gives a necessary and sufficient condition
for a family $\cF$ to be a saturating flat antichain.

\begin{lem}
\label{lem:shadow} A family $\cF=\cF_l\cup \cF_{l+1}$ with $\cF_l
\subset \binom{[n]}{l}, \cF_{l+1} \subset \binom{[n]}{l+1}$ is a
saturating antichain if and only if we have
$\Delta(\cF_{l+1})=\binom{[n]}{l} \setminus \cF_l$ and
$\nabla(\binom{[n]}{l} \setminus \cF_l)=\cF_{l+1}$.
\end{lem}
\begin{proof}
Assume first that $\cF$ is a saturating antichain. The inclusions
$\Delta(\cF_{l+1})\subseteq \binom{[n]}{l} \setminus \cF_l$ and
$\nabla(\binom{[n]}{l} \setminus \cF_l)\supseteq \cF_{l+1}$ follow
trivially from $\cF$ being an antichain. The other inclusions follow
from the saturating property of $\cF$. Indeed, if a $G \in
\binom{[n]}{l}$ is not in $\cF$, then the only reason for this is
that there should be an $F \in \cF_{l+1}$ containing $G$, similarly
if no $l$-subsets of a $(l+1)$-set $G$ belong to $\cF_l$, then $G$
can be added to $\cF_{l+1}$.

Now let us assume that $\Delta(\cF_{l+1})=\binom{[n]}{l} \setminus
\cF_l$ and $\nabla(\binom{[n]}{l} \setminus \cF_l)=\cF_{l+1}$ hold.
These equations clearly imply that $\cF$ is an antichain. Also,
$\cF$ is saturating as any $l$-set is either in $\cF$ or is
contained in a set in $\cF_{l+1}$ and any $(l+1)$-set is either in
$\cF$ or it is not in $\nabla(\binom{[n]}{l} \setminus \cF_l)$ and
therefore contains an $l$-set from $\cF$.
\end{proof}

\vskip 0.3truecm

Before we start to prove \tref{23} we need to introduce our two main
tools.

\begin{thm} [Ruzsa-Szemer\'edi, \cite{RSz}]
\label{thm:triangle} Let $G_n$ be a graph on $n$ vertices such that
the number of triangles in $G_n$ is $o(n^3)$. Then there exists a
subset $E$ of $E(G_n)$ of size $o(n^2)$ such that if we remove all
the edges in $E$ from $G_n$, then the resulting graph is triangle
free.
\end{thm}

\begin{thm} [Erd\H os-Simonovits, \cite{E,Sim}]
\label{thm:stab} Let $G_n$ be a triangle free graph on $n$ vertices
with $|E(G_n)|=(\frac{1}{4}-o(1))n^2$. Then there exists a bipartition
$V(G_n)=X \cup Y$ with $||X|-|Y||\le 1$ such that $|E(G_n)
\vartriangle E(K_{X,Y})|=o(n^2)$ holds, where $K_{X,Y}$ is the
complete bipartite graph with parts $X$ and $Y$.
\end{thm}

\begin{proof}[Proof of \tref{23}] \lref{shadow} shows that the construction
of the theorem is indeed a saturating antichain hence the upper bound of the theorem.

To prove the lower bound of the theorem let $\cF=\cF_2\cup \cF_3$
with $\cF_2 \subset \binom{[n]}{2}, \cF_3 \subset \binom{[n]}{3}$ be
a saturating antichain. Sets in $\cF_2$ and $\cF_3$ will be called
\textit{$\cF$-edges} and \textit{$\cF$-triples}, while sets in
$\binom{[n]}{2} \setminus \cF_2$ and $\binom{[n]}{3} \setminus
\cF_3$ will be called \textit{$\cF$-non-edges} and
\textit{$\cF$-non-triples}. Consider the graph $H$ of
$\cF$-non-edges, i.e. $V(H)=[n]$ and $E(H)=\binom{[n]}{2} \setminus
\cF_2$. By \lref{shadow}, we know that the triangles in $H$ are the
$\cF$-triples. As the number of $\cF$-triples is $O(n^2)=o(n^3)$ it
follows by \tref{triangle} that $H$ can be made triangle free
removing a set $E'$ of edges (i.e. $\cF$-non-edges of $\cF$) with
size $o(n^2)$. By Tur\'an's theorem \cite{T} we have that
$|E(H)|\le(\frac{1}{4}+o(1))n^2$ and thus $|\cF_2| =
\binom{n}{2}-|E(H)| \ge \binom{n}{2}
-(\frac{1}{4}+o(1))n^2=(\frac{1}{4}-o(1))n^2$. 

Let $\cF'_3 \subset \cF_3$ a maximal subfamily of the $\cF$-triples
such that for any $F_1,F_2 \in \cF'_3$ we have $|F_1 \cap F_2| \le
1$. As every $\cF$-non-edge in $E'$ is contained in at most one
$\cF$-triple in $\cF'_3$, we obtain $|\cF'_3|=o(n^2)$. By
\lref{shadow} we know that every $\cF$-non-edge  is contained in at
least one $\cF$-triple. But an $\cF$-triple covers three
$\cF$-non-edges, moreover, any $\cF$-triple in $\cF_3 \setminus
\cF'_3$ covers at most two $\cF$-non-edges that are not covered by
any $\cF$-triple in $\cF'_3$. Thus we obtain the inequality
$\binom{n}{2}-|\cF_2| \le 3|\cF'_3|+2|\cF_3 \setminus \cF'_3| \le
2|\cF_3|+o(n^2)$.

Let us define $\alpha=\alpha(n)$ by writing
$|\cF_2|=(\frac{1}{4}+\alpha)n^2$. By what we have so far, we know
that $\liminf \alpha \ge 0$. Using the inequality above we obtain
\[
|\cF|=|\cF_2|+|\cF_3| \ge
\left(\frac{1}{4}+\alpha\right)n^2+\frac{\left(\frac{1}{4}-\alpha\right)}{2}n^2-o(n^2)=\left(\frac{3}{8}+\frac{\alpha}{2}-o(1)\right)n^2.
\]
As $\liminf \alpha \ge 0$, this completes the proof of the lower
bound. Moreover, we obtain that if $\cF\subset \binom{[n]}{2} \cup
\binom{[n]}{3}$ is a saturating antichain with
$|\cF|=(\frac{3}{8}+o(1))n^2$, then we have 
$|\cF_2|=(\frac{1}{4}-o(1))n^2$ and $|\cF_3|=(\frac{1}{8}-o(1))n^2$.

All what remains is to prove the stability of the extremal family.
Note that a saturating antichain $\cF$ is clearly determined by
$\cF$-non-edges. In the case of the conjectured extremal family
these are the edges of $K_{A\cup B, C}$ and the matching $M$. Let
$\cF$ be a saturating antichain of size $(\frac{3}{8}-o(1))n^2$.
Then by what we have proved so far, we know that the graph $H$ of
the $\cF$-non-edges is of size $(\frac{1}{4}-o(1))n^2$ and is
triangle free after removing $o(n^2)$ edges. Thus, by 
\tref{stab}, after changing at most $o(n^2)$ edges in $H$ we obtain
the bipartite Tur\'an graph $K_{X,Y}$ with $||X|-|Y|| \le 1$. Let us 
put a maximal matching $M$ into any of
the two classes, say to $X$, and define the tripartition of $[n]$ as
$C$ to be the vertices not incident to $M$ and $A$ and $B$ to
contain different vertices from all edges of $M$ and $\cG$ to be the
extremal family built on this tripartition. No matter how we chose $M$, the
$\cG_2$ part of the resulting family $\cG$ will satisfy $|\cG_2
\vartriangle \cF_2|=o(n^2)$. Note that, as $X \subseteq C$ or $Y
\subseteq C$, the bipartition $A \cup B, C$ is already known up to
one vertex possibly moving from one part to the other and thus the
graph $K_{A\cup B, C}$ is known up to a change of at most $n-1$
edges.

Let $H'$ be the graph with $V(H')=[n]$ and $E(H')=E(K_{X, Y}) \cap
(\binom{[n]}{2}\setminus \cF_2)$. By the above we know that
$|E(H')|=(\frac{1}{4}-o(1))n^2$ and thus with an exception of $o(n)$
vertices every vertex has degree $(\frac{1}{2}-o(1))n$.

\begin{cl}
\label{clm:matching} Either $X$ or $Y$ contains a matching $M$ that
consists of $(\frac{1}{4}-o(1))n$ $\cF$-non-edges.
\end{cl}
\begin{proof}
We only consider vertices with degree $(\frac{1}{2}-o(1))n$ in $H'$. Note
that any $\cF$-non-edge between two such vertices in the same vertex
class defines $(\frac{1}{2}+o(1))n$ triangles in $H$ and thus, by
\lref{shadow}, that many $\cF$-triples. Also, these $\cF$-triples
are distinct, therefore there can be at most $(\frac{1}{4}+o(1))n$
such $\cF$-non-edges as we have already proved that
$|\cF_3|=(\frac{1}{8}+o(1))n^2$. Observe that all but $o(n)$
vertices in either $X$ or $Y$ are contained in at least one
$\cF$-non-edge with the other vertex in the same vertex class of
$H'$. Indeed, otherwise there would be an edge in $H'$ between such an $x
\in X$ and such a $y \in Y$ (if not, then $\Omega(n^2)$ edges would be missing from
$H'$).
And since any edge of $H'$ is an
$\cF$-non-edge, by \lref{shadow}, it has to be contained in an
$\cF$-triple all three 2-subsets of which are $\cF$-non-edges.

Now the claim follows as $(\frac{1}{4}-o(1))n$ edges can cover at
least one vertex class with the exception of $o(n)$ vertices if and only if
those edges with $o(n)$ exceptions form a matching in
one of the classes.
\end{proof}
We extend the matching $M$ given by \clref{matching} to a maximal
matching in $X$ and define the partition $A$, $B$, $C$ accordingly.
By the reasoning of \clref{matching} there are
$(\frac{1}{8}-o(1))n^2$ $\cF$-triples containing some $\cF$-non-edge
from $M$ and all these $\cF$-triples belong to $\cG$, too. As both
$|\cF_3|=(\frac{1}{8}-o(1))n^2$ and $\cG_3=(\frac{1}{8}-o(1))n^2$
hold, we obtain $|\cF_3 \vartriangle \cG_3|=o(n^2)$.
\end{proof}


In the remainder of the section we show how to generalize \tref{23}
to flat antichains with set sizes $l$ and $l+1$. Let us start by
defining the generalization of the construction of \cite{GHKLR}.

\begin{con}
\label{con:gen} Let us consider the partition $[n]=A\cup B\cup C$
with $|A|=|B|$ and let $\cM$ be a complete matching between $A$ and
$B$. Let $\cG_{l+1}=\{G \in \binom{[n]}{l+1}: \exists M \in \cM
\hskip 0.3truecm \text{with}\ M \subset G \hskip 0.3truecm
\text{and}\ G\setminus M \subset C\}$ and $\cG_l=\binom{[n]}{l}
\setminus \Delta(\cG_l)$. It is easy to see that the conditions of
\lref{shadow} hold and thus $\cG=\cG_l \cup \cG_{l+1}$ is a
saturating antichain.

The number of $(l+1)$-tuples in $\cG_{l+1}$ is
$|A|\binom{n-2|A|}{l-1}$ and the number of $l$-tuples not in $\cG_l$
is $|A|\binom{n-2|A|}{l-2}+2|A|\binom{n-2|A|}{l-1}$ thus we have
$\cG=|A|\binom{n-2|A|}{l-1}+\binom{n}{l}-(|A|\binom{n-2|A|}{l-2}+2|A|\binom{n-2|A|}{l-1})=
\binom{n}{l}-|A|\binom{n-2|A|+1}{l-1}$.
\end{con}

Observe that by replacing \tref{triangle} with the hypergraph
removal lemma \cite{G,NRS,RS}, we can use the argument of \tref{23}
to get lower bounds for the size of a saturating flat antichain
consisting only of $l$ and $(l+1)$-sets. Also, \cnref{gen} gives an
upper bound on the minimum size that such a family can have. In
order to be able to state the theorem we define $t_l$ to be the
Tur\'an-density of $K^l_{l+1}$ the complete $l$-uniform hypergraph
on $l+1$ vertices, i.e. if $ex(n,K^l_{l+1})$ denotes the most number
of edges that an $l$-uniform hypergraph on $n$ vertices can have
without containing a copy of $K^l_{l+1}$, then $t_l=\lim
ex(n,K^l_{l+1})/\binom{n}{l}$. Determining $t_l$ is one of the most
important open problems of extremal combinatorics and even the value
of $t_3$ is unknown. It is conjectured to be $5/9$ and the current
best upper bound is $\frac{3+\sqrt{17}}{12}$ \cite{CL}.

\begin{thm}
\label{thm:lflat}
\[
\left(1-\frac{l-1}{l}t_l-o(1)\right)\binom{n}{l} \le sat(n,l,l+1)
\le
\left(1-\frac{1}{2}\left(1-\frac{1}{l}\right)^{l-1}+o(1)\right)\binom{n}{l}.
\]
\end{thm}

\begin{proof}
The upper bound follows from \cnref{gen} by setting
$|A|=|B|=\frac{1}{2l}n$ and $|C|=\frac{l-1}{l}n$.
\end{proof}

Note that \tref{lflat} would not give the correct asymptotics even
in the case $l=3$ and with the assumption that Tur\'an's conjecture true.

\section{Final remarks and open problems}

In this section we enumerate the open problems in this topic that we
find the most important and interesting.
\begin{itemize}
\item
What is the correct order of magnitude of $wsat(k,k)$ and $sat(k,k)$? Do they
coincide? Can one find a sequence of families showing $sat(k,k)=o(2^k)$?
\item
We feel that there do not exist too many primitive strongly saturating
$k$-Sperner families. A better understanding of these families could help in
proving \cjref{big} via \lref{bign}.
\item
Try to close the gap between the lower and upper bounds on the minimum size of 
a saturating flat antichain for $l \ge 3$. Give any general lower bound which
does not use the Tur\'an density of hypergraphs.
\end{itemize}

\end{document}